\documentclass[12pt]{article}
\usepackage[latin1]{inputenc}
\usepackage{amsmath}
\usepackage{amsfonts}
\usepackage{amssymb}
\usepackage[american]{babel}
\usepackage{graphicx}
\usepackage{amsthm}
\usepackage{color}
\usepackage{xcolor}
\usepackage[normalem]{ulem}
\usepackage{hyperref}

\begin{document}


\title{Geometric description of $d$-dimensional flows of a graph}{}

\newtheorem{theorem}{Theorem}
\newtheorem{proposition}[theorem]{Proposition}
\newtheorem{lemma}[theorem]{Lemma}
\newtheorem{definition}[theorem]{Definition}
\newtheorem{example}[theorem]{Example}
\newtheorem{corollary}[theorem]{Corollary}
\newtheorem{conjecture}[theorem]{Conjecture}
\newtheorem{remark}{Remark}
\newtheorem{problem}{Problem}

\author{Davide Mattiolo\thanks{E-mail address: davide.mattiolo@kuleuven.be} \footnote{Department of Computer Science, KU Leuven Kulak, 8500 Kortrijk, Belgium}, Giuseppe Mazzuoccolo \footnote{Dipartimento di Informatica, Universit\`{a} degli Studi di Verona, Italy},\\Jozef Rajn\'{i}k \footnote{Department of Computer Science, Comenius University in Bratislava, Slovakia}, Gloria Tabarelli \footnote{Dipartimento di Matematica, Universit\`{a} di Trento,
Italy}}
\maketitle

\begin{abstract}
A $d$-dimensional nowhere-zero $r$-flow on a graph $G$, an $(r,d)$-NZF from now on, is a flow where the value on each edge is an element of $\mathbb{R}^d$ whose (Euclidean) norm lies in the interval $[1, r-1]$. Such a~notion is a natural generalization of the well-known concept of a circular nowhere-zero $r$-flow (i.e.\ $d = 1$). The minimum of the real numbers $r$ such that a graph $G$ admits an $(r, d)$-NZF is called the $d$-dimensional flow number of $G$ and is denoted by $\phi_d(G)$. In this paper we provide a geometric description of some $d$-dimensional flows on a graph $G$, and we prove that the existence of a suitable cycle double cover of $G$ is equivalent, for $G$, to admit such a geometrically constructed $(r,d)$-NZF. This geometric approach allows us to provide upper bounds for $\phi_{d-2}(G)$ and $\phi_{d-1}(G)$, assuming that $G$ admits an (oriented) $d$-cycle double cover.
\end{abstract}

{ {\bf Key words}: nowhere-zero flow, vector flow, cycle double cover, oriented cycled double cover

}

\section{Introduction}

The aim of this paper is to give a geometrical description of $d$-dimensional nowhere-zero $r$-flows in finite graphs and to prove connections with the (Oriented) Cycle Double Cover Conjecture.
Let $r\geq2$ be a real number and $d$ a positive integer, a \textit{$d$-dimensional nowhere-zero $r$-flow} on a graph $G$, denoted $(r,d)$-NZF on $G$, is an orientation of $G$ together with an assignment $\varphi \colon E(G)\to \mathbb{R}^d$ such that, for all $e\in E(G),$ the (Euclidean) norm of $\varphi(e)$ lies in the interval $[1,r-1]$ and, for every vertex, the sum of the inflow and outflow is the zero vector in $\mathbb{R}^d$. 
The \emph{$d$-dimensional flow number} of a bridgeless graph $G$, denoted by $\phi_d(G)$, is defined as the infimum of the real numbers $r$ such that $G$ admits an $(r,d)$-NZF. 
Note that, by Seymour's 6-flow theorem \cite{Seymour6flow} we have that  $\phi_d(G) \leq 6$ for every $d$.
Actually, $\phi_d(G)$ is a minimum: due to the above upper bound, it suffices to consider only the set of feasible $d$-dimensional nowhere-zero $r$-flows with $r \leq 6$, which can be represented as a compact subset of $\mathbb{R}^{d\cdot |E(G)|}$, and the function that assigns to every feasible flow the maximum norm among its components, that are elements of $\mathbb{R}^d$, is continuous.

When $d=1$, the major open problem in this area is Tutte's $5$-flow Conjecture \cite{Tutte5FC}, claiming that $\phi_1(G)\le5$ for all bridgeless cubic graphs $G$. Not much is known for larger values of $d$. Thomassen \cite{Tho} and Zhang et al. \cite{Zhangandal} considered flows using vectors from $S^d$, where $S^d$ denotes the set of all unit vectors from $\mathbb{R}^{d+1}$ (whose vertices form the surface of the unit sphere of dimension $d$). For brevity, an $(r, d)$-NZF using only vectors from a set $X$ is called an $X$-flow. When $d\ge3$, the following conjecture has been proposed by Jain.

\begin{conjecture}[$S^2$-flow Conjecture]\label{conj:S2}
 Let $G$ be a bridgeless graph. Then, $\phi_d(G)=2$ for every $d\ge3$.
\end{conjecture}

The case $d=2$ is already investigated by Thomassen in \cite{Tho} and then by the authors in \cite{MMRT_wheels} and \cite{MMRT_(r_d)NZF}.

In \cite{MMRT_wheels} we provide a lower bound for $\phi_2(G)$ depending on the odd-girth of $G$, when $G$ is a bridgeless cubic graphs. 
On the other hand, in \cite{MMRT_(r_d)NZF} we make a broader study of $(r,2)$-NZFs on bridgeless graphs and prove that the Oriented $5$-Cycle Double Cover Conjecture \cite{Archdeacon, Jaeger} implies that, for every bridgeless graph $G$, $\phi_2(G)\le \tau^2$, where $\tau$ denotes the Golden Ratio.

In order to introduce the Oriented 5-Cycle Double Cover Conjecture we need some terminology. If $G$ is a graph, a subgraph $H$ of $G$ is said to be a \emph{cycle} if every vertex in $H$ has even degree. If $O$ is an orientation of the edges of $G$, we denote by $O(G)$ the directed graph so obtained and, for every edge $e \in E(G)$, we denote by $O(e)$ its orientation with respect to $O$. A subgraph $H$ of $O(G)$ is a \emph{directed cycle} of $O(G)$ if $O$ induces an \emph{eulerian} orientation of $E(H)$, that is for each vertex $v$ of $H$, the indegree of $v$ equals the outdegree of $v$.

A collection $\mathcal{C}=\{C_1, \dots, C_k\}$ of cycles of a graph $G$ is said to be a \emph{cycle double cover} of $G$ if every edge $e \in E(G)$ belongs to exactly two cycles $C_i$ and $C_j$ of $\mathcal{C}$. $\mathcal{C}$ is said to be an \emph{oriented cycle double cover} of $G$ if it is a cycle double cover of $G$ with directed cycles and, for every edge $e \in E(G)$, the directions of $O_i(C_i)$ and $O_j(C_j)$ are opposite on $e$.
If we would like to stress the number of cycles in $\mathcal{C}$ we will write that $\mathcal{C}$ is a (oriented) $k$-cycle double cover of $G$ (see for instance \cite{ZhangBook}).

The Oriented $5$-Cycle Double Cover Conjecture, which is due to Archdeacon and Jaeger \cite{Archdeacon, Jaeger}, states the following.

\begin{conjecture}[Oriented $5$-Cycle Double Cover Conjecture]\label{conj:o5cdcc}
Each bridgeless graph $G$ has an oriented $5$-cycle double cover.
\end{conjecture}

Motivated by some results of Thomassen \cite{Tho}, in this note we provide a geometric interpretation of $(r,d)$-NZFs.
More precisely, in \cite{Tho} the following proposition is proved (see Proposition 1 in \cite{Tho}).

\begin{proposition}

Let $G$ be a graph. Then (a) and (b) below are equivalent,
and they imply (c) where
\item[(a)] $G$ has a nowhere zero 3-flow.
\item[(b)] $G$ has an $R_3$-flow.
\item[(c)] $G$ has an $S^1$-flow.

If $G$ is cubic the three statements are equivalent, and $G$ satisfies (a), (b), (c)
if and only if $G$ is bipartite.
\label{prop-tho}
\end{proposition}	
In particular, it is shown that a bridgeless graph $G$ has a $(3,1)$-NZF if and only if $G$ has an $R_3$-flow, i.e.\ a flow with elements in $R_3 = \{z\in\mathbb{C}\colon z^3=1\}$. This is equivalent to having an $S^1$-flow, when $G$ is cubic. 
Moreover, it is also shown that a bridgeless graph $G$ has a $3$-cycle double cover if and only if $G$ has a $T$-flow, i.e.\ a flow with elements in $T = \{(1,1,0),(0,1,1),(1,0,1),(1,-1,0),(0,1,-1),(1,0,-1)\}$. A natural generalization of $T$ to a suitable set of elements $T_d$ in $\mathbb{R}^d$ gives the equivalence between the existence of a $T_d$-flow on $G$ and a $d$-cycle double cover of $G$.
In this paper we move our attention to oriented cycle double covers. We show that, for every $d\geq 3$, a graph $G$ has an oriented $d$-cycle double cover if and only if it admits a flow with values in the set $H_d$, which is a slight variation of the set $T_d$. We also give an alternative description of $H_d$ in terms of the line graph of a crown graph.  Finally, using a geometric argument, we give an upper bound on $\phi_{d-1}(G)$ and $\phi_{d-2}(G)$, assuming that $G$ admits a $d$-cycle double cover (not necessarily oriented) and an oriented $d$-cycle double cover, respectively.

\section{Geometric description of a $d$-dimensional flow}

Let $R_k$ be the set of the $k$-th roots of unity, that is the solutions to the complex equation $z^k = 1$.
It is straightforward to see that if a graph admits an $R_3$-flow then it admits an $S^1$-flow. DeVos \cite{DeVos} suggested that the converse could also be true. Thomassen \cite{Tho} showed that this is true for cubic graphs, but does not hold in general.

Tutte's $3$-flow Conjecture claims that every $4$-edge-connected graph has a $(3,1)$-NZF.
Together with Thomassen's result this implies that every $4$-edge-connected graph has an $S^1$-flow. In our terminology this is equivalent to saying that $\phi_2(G)=2$ for every $4$-edge-connected graph $G$.

In this section, we would like to look at an $(r,d)$-NZF from a slightly different point of view. We modify a bit the notation in order to obtain an easier generalization to higher dimensions.
Instead of considering flow values in $R_3$, we consider flow values in $H=R_3 \cup -R_3$, which is clearly equivalent because of the possibility of reorienting any edge in the opposite direction if needed. The points of $H$ are the vertices of a regular hexagon in the complex plane, and so they are all points of $S^1$. 
Note that, up to rigid movements and a normalization, every nowhere-zero flow having as flow values the six vertices of an arbitrary regular hexagon centered in the origin of $\mathbb{R}^d$, can be transformed into an $H$-flow (and then in an $R_3$-flow).    

For every $d\geq 3$, we consider the following subsets of $\mathbb{R}^d$:

$$\Sigma_d=\{(x_1,\ldots,x_d)\in \mathbb{R}^d : \sum_{i=1}^d x_i = 0, \sum_{i=1}^d x_i^2=2\},$$

$$H_d=\Sigma_d \cap \mathbb{Z}^d.$$

The set $\Sigma_d$ is a sphere of dimension $d-2$ embedded in $\mathbb{R}^d$. Thus we have the following.

\begin{remark}
\label{rem:sd_equivalence}
A graph admits a $\Sigma_d$-flow if and only if it admits an $S^{d-2}$-flow.
\end{remark}

Since $H_d$ is the set of points of $\Sigma_d$ having integer coordinates, we can describe $H_d$ as the set of points of $\mathbb{R}^d$ having exactly one coordinate equal to $1$, exactly one equal to $-1$ and all remaining coordinates equal to $0$.
For instance, we have $H_3=\{(1,-1,0),(-1,1,0),(1,0,-1),(-1,0,1),(0,1,-1),$ $(0,-1,1)\}$. It is easy to check that such six points are the vertices of a regular hexagon embedded in $\mathbb{R}^3$. Then, as already observed, a graph admits an $R_3$-flow if and only if it admits an $H_3$-flow.

The notation introduced above permits to prove the following more general equivalence.

\begin{theorem}\label{thm:equiv_CDC}
A graph $G$ admits an $H_d$-flow if and only if $G$ admits an oriented $d$-cycle double cover. 
\end{theorem}
\begin{proof}
Let $\mathcal{C}=\{O_1(C_1),\dots,O_d(C_d)\} $ be an oriented $d$-cycle double cover of $G$. Choose an arbitrary orientation $O$ for the graph $G$ and consider $e \in E(G)$. Since $\mathcal{C}$ is an oriented $d$-cycle double cover of $G$ there exist exactly two different indices $h, k \in \{1,2,...,d\}$ such that $e \in C_h\cap C_k$ and $O_h(e) = O(e) \neq O_k(e)$. We assign to the edge $e$ the $d$-tuple $\varphi(e)$ having $1$ in the entry $h$, $-1$ in the entry $k$ and $0$ in all other entries. Note that, for $i\in\{1,\dots,d\}$, the $i$-th component of $\varphi$ defines a nowhere-zero $2$-flow on $C_i$ with respect to $O$. Hence, $\varphi$ is an $H_d$-flow on $G$ with respect to the chosen orientation $O$.

For the converse, let $\varphi$ be an $H_d$-flow on $G$ with respect to the orientation $O$ of $G$. Construct an oriented $d$-cycle double cover of $G$ as follows.
For each $i\in \{1,...,d\}$, let $C_i$ be the subgraph of $G$ induced by the edges $e \in E(G)$ such that $\varphi(e) \neq 0$ in the $i$-th entry. Note that $C_i$ is a cycle of $G$. Indeed, since $\varphi$ is an $H_d$-flow on $G$, for every vertex $v\in V(G)$ there is an even number, eventually $0$, of edges incident to $v$ with a non-zero value of $\varphi$ in the $i$-th entry.
Construct an orientation $O_i$ on $C_i$ by assigning $O_i(e)=O(e)$ on every edge $e$ of $C_i$ such that $\varphi(e)=1$, and letting $O_i(e)$ be opposite to $O(e)$ otherwise.  
By construction of $O_i$, for every vertex $v\in V(C_i)$, the indegree of $v$ with respect to $O_i$ is equal to the outdegree of $v$ with respect to $O_i$, since $\varphi$ is an $H_d$-flow on $G$ with respect to the orientation $O$. Hence, $O_i(C_i)$ is a directed cycle of $G$. Moreover, note that, for every edge $e \in E(G)$, $\varphi(e)$ has exactly two non-zero entries and that such entries have opposite values. Therefore, the collection $\mathcal{C}=\{O_1(C_1) \dots, O_d(C_d)\}$ is an oriented $d$-cycle double cover of $G$.
\end{proof}

The following result is a consequence of Proposition \ref{prop-tho} and previous results.

\begin{corollary} \label{cor_d3}
Let $G$ be a graph. The following assertions are equivalent: 
\begin{enumerate}
 \item $G$ has a nowhere-zero $3$-flow;
 \item $G$ has an $R_3$-flow;
 \item $G$ has an $H_3$-flow;
 \item $G$ has an oriented $3$-cycle double cover.
\end{enumerate}
\end{corollary}

Hence, in our terminology, Thomassen shows that the existence of an $H_3$-flow for a graph $G$ implies the existence of a $\Sigma_3$-flow for a graph $G$. Notice that, from our definition of $H_3$, this implication is straightforward. In \cite{Tho} it is shown that the converse is not true unless $G$ is a cubic graph. Indeed, as already remarked, Thomassen presented examples of graphs admitting a $\Sigma_3$-flow but without an $H_3$-flow.

The Petersen graph $P$ is an example of a graph that admits an  $\Sigma_4$-flow but, since it is not $3$-edge-colourable, without an oriented $4$-cycle double cover, and hence without an $H_4$-flow.
Indeed, observe that $P$ admits a $\Sigma_4$-flow since it admits an $S^2$-flow (Remark \ref{rem:sd_equivalence}), which is depicted in Figure \ref{fig:S2_flow_pet}. In this figure, the vectors $z_1$, $z_2$, $z_3$, $z_4$, and $z_5$ are unit vectors. They are arranged to form a pentagon and a star lying in parallel planes and having their vertices on the unitary sphere.

\begin{figure}[h]
\begin{minipage}[l]{0.45\linewidth}
\includegraphics[width=\linewidth]{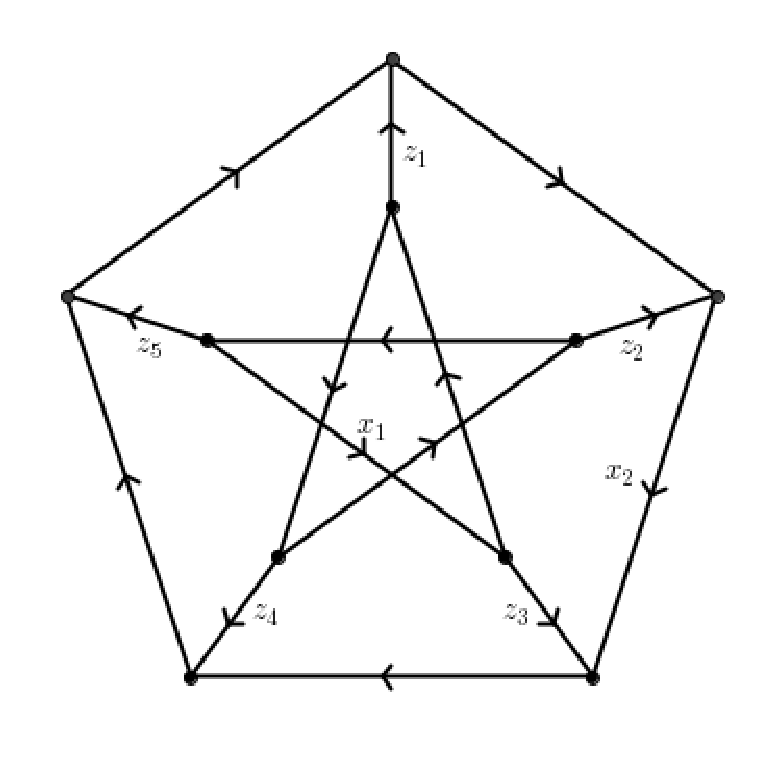}
\end{minipage}
\hfill
\begin{minipage}[l]{0.45\linewidth}
\includegraphics[width=\linewidth]{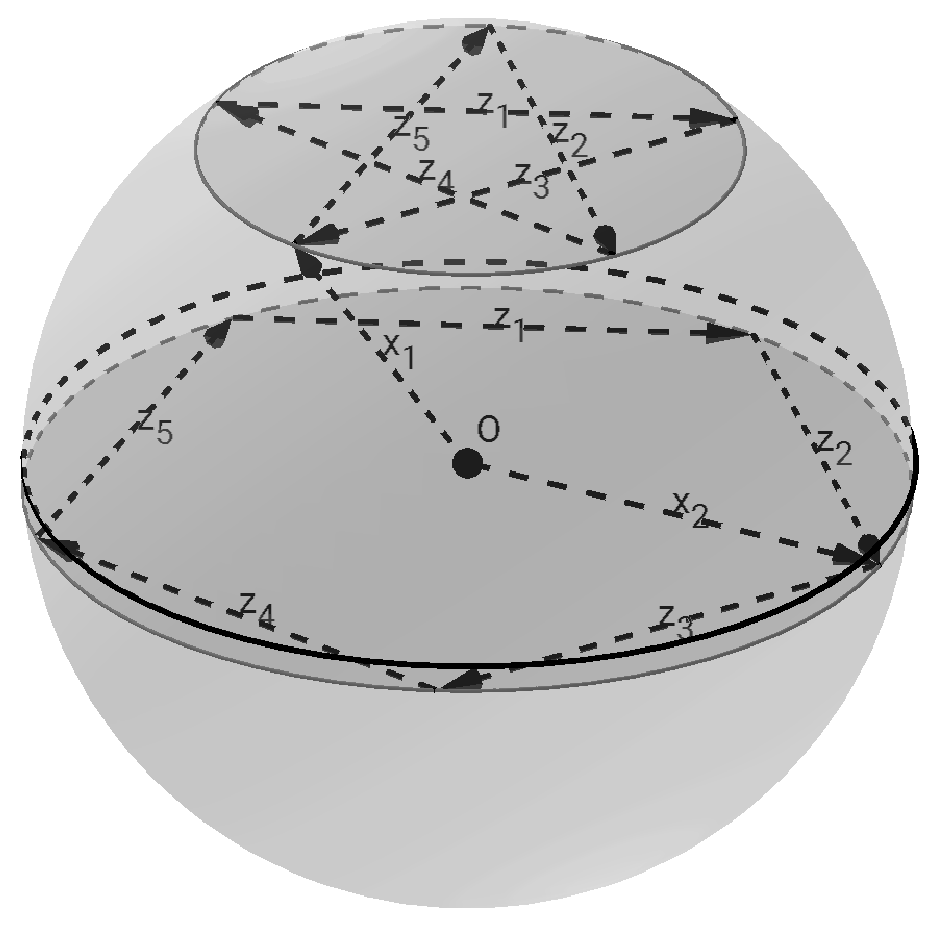}
\end{minipage}%
\caption{An $S^2$-flow on the Petersen Graph.}
\label{fig:S2_flow_pet}
\end{figure}

By previous considerations it seems natural to ask whether there exists a graph $G$ having a $\Sigma_5$-flow but without an $H_5$-flow. We remark that, such graph might not even exist as, by Theorem \ref{thm:equiv_CDC}, it would be a counterexample to the Oriented $5$-Cycle Double Cover Conjecture.

\subsection{A description of $H_d$ as vertex set of a line graph}

Note that $H_4$ can be viewed as the set of vertices of a cuboctahedron (see \cite{CoxeterBook}) embedded in $\mathbb{R}^4$. Analogously, $H_5$ is the set of vertices of a runcinated $5$-cell in $\mathbb{R}^5$ (see \cite{5-cell} for a definition), and more in general polytopes of dimension $d-1$ embedded in $\mathbb{R}^d$ can be obtained, for every set $H_d$, by connecting two elements of $H_d$ with an edge if and only if their difference is still an element of $H_d$. For every $d\geq 4$, all these polytopes can be viewed as a sort of generalization of the regular hexagon: indeed, as it happens in the regular hexagon, all their edges have the same length and such length is equal to the distance between the origin and each vertex.

In terms of graphs, we prove that the graph $G_d$ corresponding to the $(d-1)$-dimensional polytope described above is nothing but the line graph of the crown graph $Cr(2d)$ on $2d$ vertices. Recall that the crown graph $Cr(2d)$ can be described as a complete bipartite graph from which the edges of a perfect matching have been removed.  

\begin{proposition}
Consider the graphs $G_d$ and $Cr(2d)$ described above. Then $G_d \cong L(Cr(2d))$.
\end{proposition}
\begin{proof}
By definition of $G_d$ we have $V(G_d)=H_d$ and $E(G_d)=\{ab \mid a,b \in H_d,\,\,\, \pm (a-b) \in H_d\}$.
Denote the vertices of $Cr(2d)$ by $u_1,...,u_d,v_1,...v_d$, where the vertices $u_i$ are in an independence class of $Cr(2d)$ and the vertices $v_i$ in the other one. Let $E(Cr(2d))= \{u_iv_j \mid i \neq j\}$ and use the same notation for the edges of $Cr(2d)$ and the vertices of its line graph.

Associate to any $u_iv_j \in V(L(Cr(2d)))$ the $d$-tuple with $1$ in the $i$-th entry, $-1$ in the $j$-th entry and $0$ in all the other entries. This association gives a bijective map between the vertices of $L(Cr(2d))$ and $H_d=V(G_d)$.
Observe that a vertex $u_iv_j \in V(L(Cr(2d)))$ is adjacent to any other vertex $u_lv_m$ such that either $i=l$ or $j=m$. Moreover, the edge-set of $G_d$ can be also described as the set of (unordered) pairs $(a,b)$ such that $a,b \in H_d$ and $a$ and $b$ are equal in exactly one non-zero entry. Hence the bijective map described above induces a bijection also between $E(L(Cr(2d)))$ and $E(G_d)$.
\end{proof}

\subsection{Upper bounds for $\phi_{d-1}(G)$ and $\phi_{d-2}(G)$}

In \cite{MMRT_(r_d)NZF} it is recalled that if a graph $G$ admits a $5$-cycle double cover, then $\phi_5(G)=2$. If we make a stronger hypothesis on $G$, that is that $G$ admits an oriented $5$-cycle double cover, then we obtain that $\phi_4(G)=2$. Indeed, by Theorem \ref{thm:equiv_CDC}, this is equivalent, for $G$, to have an $H_5$-flow. Hence, $G$ admits also a $\Sigma_5$-flow, and by Remark \ref{rem:sd_equivalence}, also an $S^3$-flow. In particular, it holds that $\phi_4(G)=2$.

Hence the following holds true.

\begin{proposition}
The Oriented $5$-Cycle Double Cover Conjecture (Conjecture \ref{conj:o5cdcc}) implies the $S^2$-flow Conjecture (Conjecture \ref{conj:S2}) for $d\geq 4$.
\end{proposition}  

More in general, arguing as above, it holds that if a graph $G$ has an oriented $d$-cycle double cover, it has an $S^{d-2}$-flow, implying that $\phi_{d-1}(G)=2$.

If we assume $G$ to only admit a $d$-cycle double cover, not necessarily oriented, we obtain the following upper bound on $\phi_{d-1}(G)$.


\begin{proposition}
For every  $d \ge 3$, if a graph $G$ has a $d$-cycle double cover $\mathcal{C}$, then $\phi_{d - 1}(G) \le 1 + \sqrt{d/(d-2)}$.
\label{prop:upper_d-1}
\end{proposition}

\begin{proof}
We can assume that $\mathcal{C} = \{C_1,\dots,C_d\}$.
We now consider $d$ points $a_1,..., a_d \in \mathbb{R}^{d-1}$. More precisely, for $i \in \{1, \dots, d - 1\}$, let $a_i \in \mathbb{R}^{d-1}$ have the value $d\sqrt{d} - 2\sqrt{d} - 1$ in its $i$-th entry and $- \sqrt{d} - 1$ in all the other entries. Let $a_d \in \mathbb{R}^{d-1}$ have $d - 1$ in all its entries.
Let $O$ be an orientation of $G$ and, for each $i\in\{1,\dots,d\}$, fix an eulerian orientation $O_i$ on $C_i$.

For all $i\in\{1,\dots,d\}$, add the flow value $a_i$, resp.\ $-a_i$, to all edges $e\in C_i$ such that $O_i(e)=O(e)$, resp.\ $O_i(e)\ne O(e).$

Note that every edge of $G$ is contained in exactly two members of $\mathcal{C}$.
Then, every edge of $G$ receives a vector with a norm  $|a_i + a_j| = (d - 1)\sqrt{2(d-2)}$ or $|a_i - a_j| = (d - 1)\sqrt{2d}$, for $1 \le i \ne j \le d$. After normalizing, we obtain a $(1 + \sqrt{d/(d-2)}, d - 1)$-NZF on $G$.
\end{proof}

Observe that the upper bound of Proposition \ref{prop:upper_d-1} approaches $2$ as $d$ grows.
We give a similar result for $\phi_{d-2}$.

In \cite{BatErd}, the authors claim that very likely the ratio between the maximum and the minimum distance for a set of $d$ points in $R^{d-2}$ is larger than $\sqrt{4/3}$. That was proved to be false by Seidel in 1969. Example 2.3 and Example 2.4 in \cite{Sei} are proved to be optimal in order to minimize such a ratio. By using Seidel's examples we obtain the following.

\begin{proposition}
For every $d\ge 3$, if a graph $G$ admits an oriented $d$-cycle double cover $\mathcal{C}$, then $$\phi_{d-2}(G)\leq \begin{cases} 1+\sqrt{\frac{d}{d-2}} \text{ if $d$ even,} \\
1+\sqrt{\frac{d^2-1}{d^2-2d-1}} \text{ if $d$ odd.}
\end{cases}$$.
\label{prop:upper_d-2}
\end{proposition}
\begin{proof}
If $d=3$ the statement follows from Corollary \ref{cor_d3}.

Let $d\ge4$ and let $\mathcal{C} = \{C_1,\dots,C_d\}$. Set $d_1\leq d_2$ such that $d_1+d_2=d-2$ and $|d_1-d_2|\leq1$.
Let $U_1$ and $U_2$ be two orthogonal and complementary subspaces of $\mathbb{R}^{d-2}$ passing through the origin, and having dimensions $d_1$ and $d_2$, respectively.
For $i=1,2$, consider a regular $d_i$-dimensional simplex $F_i$ in $U_i$ and centered in the origin. Choose $F_1$ and $F_2$ having the same side length equal to $\sqrt{d-2}$ if $d$ even, and $\sqrt{d-\frac{1}{d}-2}$, if $d$ odd.

Denote by $P_1,P_2,...,P_{d}$ the union of the vertices of $F_1$ and $F_2$.

As in the proof of Proposition \ref{prop:upper_d-1}, let $O$ be an orientation of $G$ and, for each $i\in\{1,\dots,d\}$, fix an eulerian orientation $O_i$ on $C_i$.
For all $i\in\{1,\dots,d\}$, add the flow value $P_i$, resp.\ $-P_i$, to all edges $e\in C_i$ such that $O_i(e)=O(e)$, resp.\ $O_i(e)\ne O(e).$

If $d$ is even, see also Example 2.3 in \cite{Sei}, the distance between two points $P_i$ and $P_j$ is either $\sqrt{d}$ or $\sqrt{d-2}$.
If $d$ is odd, see also Example 2.4 in \cite{Sei}, the distance between two points $P_i$ and $P_j$ is either $\sqrt{d-\frac{1}{d}}$ or $\sqrt{d-\frac{1}{d}-2}$.
Thus, in the former case the ratio between maximum and minimum norm of flow values is $\sqrt{d/d-2}$ and in the latter case is $\sqrt{(d^2-1)/(d^2-2d-1)}$.
\end{proof}

\section*{Acknowledgments}

Davide Mattiolo is supported by a Postdoctoral Fellowship of the Research Foundation Flanders (FWO), project number 1268323N.

\section*{Statements and Declarations}

The authors have no relevant financial or non-financial interests to disclose.

\section*{Data availability} Data sharing not applicable to this article as no datasets were generated or analyzed during the current study.

\end{document}